\documentclass[sn-basic]{sn-jnl}

\usepackage{graphicx}%
\usepackage{multirow}%
\usepackage{amsmath,amssymb,amsfonts}%
\usepackage{amsthm}%
\usepackage{mathrsfs}%
\usepackage[title]{appendix}%
\usepackage{xcolor}%
\usepackage{textcomp}%
\usepackage{manyfoot}%
\usepackage{booktabs}%
\usepackage{algorithm}%
\usepackage{algorithmicx}%
\usepackage{algpseudocode}%
\usepackage{listings}%

\theoremstyle{thmstyleone}%
\newtheorem{theorem}{Theorem}%
\newtheorem{corollary}[theorem]{Corollary}%
\newtheorem{lemma}[theorem]{Lemma}%
\newtheorem{proposition}[theorem]{Proposition}%
\theoremstyle{thmstyletwo}%
\newtheorem{example}{Example}%
\theoremstyle{thmstylethree}%

\usepackage{tikz, tikz-3dplot}
\usepackage{tikz-network}
\usetikzlibrary{positioning, arrows.meta}

\renewcommand{\Re}{\mathbb{R}}
\renewcommand{\b}[1]{\boldsymbol{#1}}
\newcommand{\cA}{\mathcal{A}}
\newcommand{\cN}{\mathcal{N}}
\DeclareMathOperator{\conv}{conv}
\DeclareMathOperator{\rank}{rank}

\begin{document}
\title[Convexification of a bilinear set]{On the convexification of a bilinearly constrained set with a common variable}
\author[1]{\fnm{Hyun-Ju} \sur{Oh}}\email{hoh@clemson.edu}

\author*[1]{\fnm{Boshi} \sur{Yang}}\email{boshiy@clemson.edu}

\affil*[1]{\orgdiv{School of Mathematical and Statistical Sciences}, \orgname{Clemson University}, \orgaddress{\street{220 Parkway Dr}, \city{Clemson}, \postcode{29634}, \state{South Carolina}, \country{United States}}}

\date{\today}
\abstract{We study the convexification of sets defined by box constraints and bilinear inequality constraints in which all bilinear terms share a common variable. For this structured class of nonconvex sets, we provide a complete characterization of the extreme points of the convex hull and derive an explicit description of all facet-defining inequalities, thereby obtaining a full polyhedral description of the convex hull. In addition, we develop a polynomial-time separation algorithm for the facet-defining inequalities and illustrate its effectiveness through examples. These results offer both structural insight into the geometry of the convex hull and algorithmic tools for identifying violated inequalities.}%

\keywords{Nonconvex optimization, Bilinear programming, Convex hull, Global optimization}

\maketitle
\section{Introduction}

Optimization problems with bilinear terms in their objective functions or constraints arise in many applications and are generally nonconvex. A central challenge in solving such problems is the construction of strong convex relaxations of the feasible region. Classical and widely used approaches include the McCormick envelope \citep{mccormick1976computability, al1990jointly, quesada1995global}, which gives the convex hull of the graph of a single bilinear term over a rectangular domain, and the Reformulation-and-Linearization Technique (RLT) \citep{sherali1990hierarchy, sherali1992new}, which generates valid linear inequalities by multiplying constraints and linearizing the resulting products. However, when several bilinear terms share a common variable, applying these relaxations term by term can overlook valid inequalities that couple the products through the common factor.

Motivated by this gap, we study a basic shared-variable set that already exhibits nontrivial coupling. After rescaling and relabeling if necessary, let $\b{a}\in\Re^n$ satisfy $0<a_1\leq\cdots\leq a_n<1$, and define
\begin{equation}\label{eqn:S}
    S^n(\b{a})
    :=
    \left\{
        (x,\b{y})\in[0,1]^{n+1}
        \;\middle|\;
        xy_j\leq a_j \quad \forall j\in[n]
    \right\}.
\end{equation}
When no ambiguity can arise, we write $S^n$ in place of $S^n(\b{a})$. The defining feature of $S^n$ is that all bilinear inequalities share the variable $x$, while each product is subject to an individual nontrivial upper bound. Although the description of $S^n$ is simple, its convex-hull geometry is not. When $n=1$, the projected McCormick relaxation is exact: together with the box constraints, the inequality
\[
    x+y_1\leq 1+a_1
\]
describes $\conv(S^1)$. For $n\geq2$, by contrast, the term-wise projected McCormick relaxation need not be exact. The shared variable can induce additional inequalities that couple several components of $\b{y}$. For example, when $n=2$ and $a_1<a_2$, the inequality
\[
    x+a_2y_1+(1-a_1)y_2\leq1+a_2
\]
is valid for $\conv(S^2)$ but is not implied by the two individual projected McCormick inequalities. Fig.~\ref{fig:conv(S2)} illustrates the resulting difference for $a_1=1/4$ and $a_2=1/2$.

\begin{figure}[htbp]
  \centering
  \begin{tabular}{cc}
    \tdplotsetmaincoords{60}{110}
    \begin{tikzpicture}[scale=2.5,tdplot_main_coords]
      \begin{scope}
        \coordinate (O) at (0,0,0);
        \draw[-stealth] (0,0,0) -- (1.2,0,0) node[left]{\small$x$};
        \draw[-stealth] (0,0,0) -- (0,1.2,0) node[right]{\small$y_1$};
        \draw[-stealth] (0,0,0) -- (0,0,1.2) node[above]{\small$y_2$};
        \draw[thick] (1, 1/4, 0) -- (1, 0, 0) -- (1, 0, 1/2);
        \draw[thick] (0, 0, 1) -- (1/2, 0, 1);
        \draw[thick] (0, 0, 1) -- (0, 1, 1) -- (1/4, 1, 1) -- (0, 1, 1) -- (0, 1, 0) -- (1/4, 1, 0);
        \draw[ultra thick] plot[tension=1] (1, 0, 1/2) -- (1/2, 0, 1) -- (1/2, 3/4, 1) -- (1, 0.25, 0.5) -- cycle;
        \draw[ultra thick] plot[tension=1] (1, 0.25, 0.5) -- (1, 0.25, 0) -- (0.25, 1, 0) -- (0.25, 1, 1) -- (1/2, 3/4, 1);
        \draw[ultra thick] plot[tension=1] (1, 0, 0.5) -- (1, 0, 0) -- (1, 1/4, 0);
        \draw[ultra thick] plot[tension=1] (1/4, 1, 0) -- (0, 1, 0) -- (0, 1, 1) -- (1/4, 1, 1);
        \draw[ultra thick] plot[tension=1] (1/2, 0, 1) -- (0, 0, 1) -- (0, 1, 1);
        \draw[fill = blue, fill opacity = 0.3] (1/2, 0, 1) -- (1/2, 3/4, 1) -- (1, 1/4, 1/2) -- (1, 0, 1/2) -- (1/2, 0, 1);
        \draw[fill=blue, fill opacity = 0.3] (1, 1/4, 1/2) -- (1, 1/4, 0) -- (1/4, 1, 0) -- (1/4, 1, 1) -- (1/2, 3/4, 1) -- (1, 1/4, 1/2);
        \draw[fill = blue, fill opacity = 0.3] (1, 0, 0) -- (1, 0, 1/2) -- (1, 1/4, 1/2) -- (1, 1/4, 0) -- cycle;
        \draw[fill = blue, fill opacity = 0.3] (0, 0, 1) -- (1/2, 0, 1) -- (1/2, 3/4, 1) -- (1/4, 1, 1) -- (0, 1, 1) -- cycle;
        \draw[fill = blue, fill opacity = 0.3] (0, 1, 0) -- (1/4, 1, 0) -- (1/4, 1, 1) -- (0, 1, 1) -- cycle;
      \end{scope}
    \end{tikzpicture} 
    & 
    \tdplotsetmaincoords{60}{110}
    \begin{tikzpicture}[scale=2.5,tdplot_main_coords]
      \begin{scope}
        \coordinate (O) at (0,0,0);
        \draw[-stealth] (0,0,0) -- (1.2,0,0) node[left]{\small$x$};
        \draw[-stealth] (0,0,0) -- (0,1.2,0) node[right]{\small$y_1$};
        \draw[-stealth] (0,0,0) -- (0,0,1.2) node[above]{\small$y_2$};
        \draw[thick] (1, 1/4, 0) -- (1, 0, 0) -- (1, 0, 1/2);
        \draw[thick] (0, 0, 1) -- (1/2, 0, 1);
        \draw[thick] (0, 0, 1) -- (0, 1, 1) -- (1/4, 1, 1) -- (0, 1, 1) -- (0, 1, 0) -- (1/4, 1, 0);
        \draw[ultra thick] plot[tension=1] (1, 0, 1/2) -- (1/2, 0, 1) -- (1/2, 1/2, 1) -- (1, 0.25, 0.5) -- cycle;
        \draw[ultra thick] plot[tension=1] (1, 0.25, 0) -- (0.25, 1, 0) -- (0.25, 1, 1) -- (1, 0.25, 0.5) -- cycle;
        \draw[ultra thick] plot[tension=1] (0.5, 0.5, 1) -- (0.25, 1, 1);
         \draw[ultra thick] plot[tension=1] (1, 0, 0.5) -- (1, 0, 0) -- (1, 1/4, 0);
        \draw[ultra thick] plot[tension=1] (1/4, 1, 0) -- (0, 1, 0) -- (0, 1, 1) -- (1/4, 1, 1);
        \draw[ultra thick] plot[tension=1] (1/2, 0, 1) -- (0, 0, 1) -- (0, 1, 1);
        \draw[fill = blue, fill opacity = 0.3] (1/2, 0, 1) -- (1/2, 1/2, 1) -- (1, 1/4, 1/2) -- (1, 0, 1/2) -- (1/2, 0, 1);
        \draw[fill=blue, fill opacity = 0.3] (1, 1/4, 1/2) -- (1, 1/4, 0) -- (1/4, 1, 0) -- (1/4, 1, 1) -- (1, 1/4, 1/2);
        \draw[fill=blue, fill opacity = 0.3] (1, 1/4, 1/2) -- (1/4, 1, 1) -- (1/2, 1/2, 1) -- (1, 1/4, 1/2);
        \draw[fill = blue, fill opacity = 0.3] (1, 0, 0) -- (1, 0, 1/2) -- (1, 1/4, 1/2) -- (1, 1/4, 0) -- cycle;
        \draw[fill = blue, fill opacity = 0.3] (0, 0, 1) -- (1/2, 0, 1) -- (1/2, 1/2, 1) -- (1/4, 1, 1) -- (0, 1, 1) -- cycle;
        \draw[fill = blue, fill opacity = 0.3] (0, 1, 0) -- (1/4, 1, 0) -- (1/4, 1, 1) -- (0, 1, 1) -- cycle;
      \end{scope}
    \end{tikzpicture} \\
  \end{tabular}
    \caption{The projected McCormick relaxation (left) and exact convex hull (right) of $S^2$ when $a_1 = \frac14$ and $a_2 = \frac12$.}
  \label{fig:conv(S2)}
\end{figure}
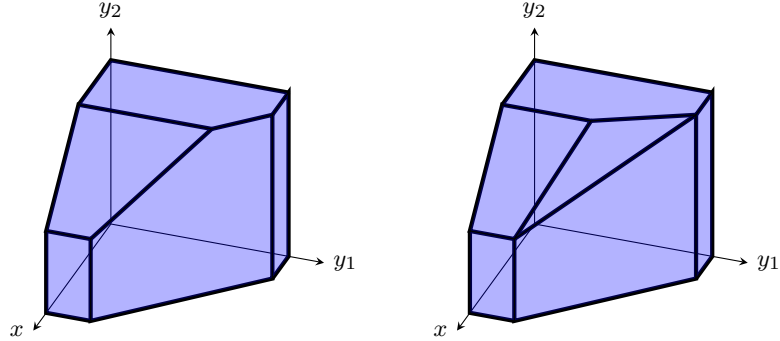

Sets of the form \eqref{eqn:S} arise naturally as local structures in bipartite bilinear programs, in which the variables are partitioned into two groups and products involve variables from different groups \citep{dey2019new}. Focusing on the products incident to a single variable yields the shared-variable pattern studied here, particularly when individual upper bounds on those products are available. Related bilinear structures occur in models for wastewater management \citep{galan1998optimal, faria2011planning, teles2012global, castro2015tightening}, farm management \citep{bloemhof1996generalized}, power distribution planning \citep{frank2015optimal}, pooling \citep{haverly1978studies, foulds1992bilinear, tawarmalani2013convexification, boland2016new, gupte2017relaxations, chen2024tightening}, supply chain management \citep{al1992generalized, nahapetyan2008bilinear, osman2010bilinear}, and network interdiction \citep{davarnia2017simultaneous, azizi2024shortest}. Thus, $S^n$ serves as a basic building block for strengthening relaxations of broader classes of nonconvex optimization models.

Our work is closely related to the convexification of bilinear and multilinear product sets. Convex-hull and envelope descriptions in spaces containing auxiliary product variables have been developed for a single bilinear product and for more general bounded bilinear monomials \citep{belotti2010valid, belotti2011linear, nguyen2018deriving, anstreicher2021convex, chen2022tightening}. For bounded monomials on two-variable cones, \cite{belotti2025convex} derives upper envelopes in general dimension and lower envelopes in the bivariate case. A closely related study by \citet{khademnia2025convexification} develops convexification results for bilinear products over a simplex and a network polytope. Specializing their setting to
a one-dimensional simplex and a box yields the set
\[
    G^n := \left\{(x,\b{y},\b{w}) \in [0,1]^{n+1} \times \Re^n \;\middle|\;
        w_i=xy_i \quad \forall \, i \in [n] \right\}.
\]
The set $S^n$ in \eqref{eqn:S} is obtained by imposing $w_i\leq a_i$ for all $i\in[n]$ and projecting onto the $(x,\b{y})$-space. For this specialization, $\conv(G^n)$ is described by the McCormick inequalities. Consequently, applying their convexification to $G^n$ and then imposing the product bounds recovers the McCormick relaxation with those bounds, whose projection need not equal $\conv(S^n)$. Our approach instead incorporates the product bounds into the convexification and gives a complete description of $\conv(S^n)$ in the original $(x,\b{y})$-space. The resulting inequalities capture interactions among the bounded products
sharing $x$ and can strengthen lifted relaxations of bilinear programs beyond the McCormick inequalities.

This paper makes three main contributions. First, we characterize all extreme points of $\conv(S^n)$ and thereby establish that the convex hull is a polytope. In particular, every nonzero extreme point belongs to a face obtained by fixing either $x=1$ or $y_i=1$ for some $i\in[n]$. Second, we derive a complete facet description of $\conv(S^n)$. Beyond the variable bounds, the nontrivial facets form a structured family of \emph{leap inequalities} that couple multiple bilinear constraints through the ordered upper bounds $a_1,\ldots,a_n$. Third, we show that a most violated leap inequality can be identified by solving a longest-path problem on a directed acyclic graph, yielding an $\mathcal O(n^2)$ separation algorithm without enumerating the entire family. Finally, two examples illustrate how leap inequalities can complement second-order cone constraints in strengthening convex relaxations of bipartite bilinear programs.

The remainder of the paper is organized as follows. Section~\ref{sec:extpts} characterizes the extreme points of $\conv(S^n)$. Section~\ref{sec:facets} derives its complete facet description. Section~\ref{sec:sep} develops the separation algorithm and presents two examples illustrating the strengthening obtained from leap inequalities, both alone and in combination with the second-order cone relaxation studied in \cite{anstreicher2021convex}.

\paragraph{Notation.}
Vectors are denoted by boldface lowercase letters and matrices by uppercase letters. For a positive integer $n$, let $[n]:=\{1,\ldots,n\}$. Given $\b{z}\in\Re^n$ and $k\in[n]$, let $\b{z}_{1:k}\in\Re^k$ denote the subvector consisting of the first $k$ components of $\b{z}$. 
For $j\in[n]$, $\b{e}_j\in\Re^n$ denotes the $j$th standard unit vector. The convex hull of a set $C$ is denoted by $\conv(C)$. For $\b{a},\b{b}\in\Re^n$, their Hadamard product is
\[
    \b{a}\circ\b{b}
    :=
    (a_1b_1,\ldots,a_nb_n).
\]

\section{The vertex description of $\conv(S^n)$} \label{sec:extpts}
In this section, we characterize the extreme points of $\conv(S^n)$. To this end, we first define $n+1$ subsets of $S^n$ by fixing one variable at a time as follows.
\[
C^n_i := \begin{cases}
    \{(x, \b{y}) \in S^n \mid x = 1\} & \text{ if } i = 0, \\
    \{(x, \b{y}) \in S^n \mid y_i = 1\} & \text{ if } i \in [n].
\end{cases}
\]
We also note that $S^n$ is full-dimensional. Indeed, since $0<a_j<1$ for all $j\in[n]$, the point $(\frac12, a_1, \ldots, a_n)$ belongs to the interior of $S^n$. Consequently, $\conv(S^n)$ is also full-dimensional.

\subsection{Extreme points of $\conv(C^n_i)$}
We first characterize the extreme points of $\conv(C^n_i)$ for each $i\in\{0,\ldots,n\}$. Let $E^n_i$ denote the set of extreme points of $\conv(C^n_i)$. Since $C^n_i$ is compact, every extreme point of its convex hull belongs to $C^n_i$, and thus $E^n_i\subseteq C^n_i$.

For $i=0$, we have $C^n_0=\{1\}\times\prod_{j=1}^n[0,a_j]$, and hence
\[
    E^n_0 =\left\{\left(1,\b{b}\circ(a_1,\ldots,a_n)\right) \middle|\ \b{b}\in\{0,1\}^n\right\}.
\]
Now fix $i\in[n]$. For each $k\in[i]$, let
\[
    \b{\gamma}^{\,k} :=\left(\frac{a_1}{a_k},\ldots,\frac{a_k}{a_k},1,\ldots,1\right) \in[0,1]^n.
\]
Define
\begin{align*}
    D^n_{i,0} &:={\left\{(0,\b{b}) \middle|\ \b{b}\in\{0,1\}^n,\ b_i=1\right\}},\\
    D^n_{i,k} &:={\left\{\left(a_k,\b{b}\circ\b{\gamma}^{\,k}\right) \middle|\ \b{b}\in\{0,1\}^n,\ b_i=b_k=1\right\}}, \qquad k\in[i].
\end{align*}
That is, a point $(a_k,\b{y})$ belongs to $D^n_{i,k}$ precisely when $y_i=y_k=1$, each $y_j$ with $j\leq k$ belongs to $\{0,a_j/a_k\}$, and each $y_j$ with $j>k$ belongs to $\{0,1\}$.

We first show that all points in these sets are extreme points of $\conv(C^n_i)$.

\begin{lemma} \label{lem:D0}
    For every $i \in [n]$, $D^n_{i,0} \subseteq E^n_i$.
\end{lemma}

\begin{proof}
    Clearly, $D^n_{i,0} \subseteq C^n_i \subseteq \conv(C^n_i) \subseteq [0,1]^{1+n}$. Since the points in $D^n_{i,0}$ are extreme points of $[0,1]^{1+n}$, they are also extreme points of $\conv(C^n_i)$.
\end{proof}

\begin{lemma}\label{lem:Dk}
    For every $i \in [n]$ and $k \in [i]$, $D^n_{i,k} \subseteq E^n_i$.
\end{lemma}

\begin{proof}
    Fix $i \in [n]$, $k \in [i]$, and let $(\bar{x},\bar{\b{y}}) \in D^n_{i,k}$.
    By the definition of $D^n_{i,k}$, we have $\bar{x}=a_k$, $\bar{y}_k=1$, and, for every $s \in [n] \setminus \{k\}$,
    \[
        \bar{y}_s \in
        \begin{cases}
            \left\{0,\dfrac{a_s}{a_k}\right\}, & s < k,\\[2mm]
            \{0,1\}, & s>k.
        \end{cases}
    \]
    Suppose that
    \[
        (\bar{x},\bar{\b{y}}) = \lambda(x^1,\b{y}^1) +(1-\lambda)(x^2,\b{y}^2)
    \]
    for some $(x^1,\b{y}^1),(x^2,\b{y}^2)\in\conv(C^n_i)$ and $\lambda\in(0,1)$. For each $t\in[2]$, write
    \[
        (x^t,\b{y}^t) = \sum_{r\in R_t}\rho^t_r (x^{t,r},\b{y}^{t,r}),
    \]
    where $(x^{t,r},\b{y}^{t,r})\in C^n_i$, $\rho^t_r>0$, and $\sum_{r\in R_t}\rho^t_r=1$.
    Thus, $(\bar{x},\bar{\b{y}})$ is a convex combination, with strictly positive coefficients, of the points $(x^{t,r},\b{y}^{t,r})\in C^n_i$.
    Since $0\leq y^{t,r}_k\leq 1$ for every $t$ and $r$, while $\bar{y}_k=1$, it follows that $y^{t,r}_k=1$ for every $t\in[2]$ and $r\in R_t$.
    Because $(x^{t,r},\b{y}^{t,r})\in C^n_i\subseteq S^n$, we have $x^{t,r}y^{t,r}_k\leq a_k$. Hence $x^{t,r}\leq a_k$. On the other hand, $\bar{x}=a_k$ is a convex combination of the values $x^{t,r}$. Since all these values are at most $a_k$ and all coefficients are strictly positive, we
    must have $x^{t,r}=a_k$ for every $t\in[2]$ and $r\in R_t$.
    Now fix $s\in[k]$. Since $(x^{t,r},\b{y}^{t,r})\in S^n$ and $x^{t,r}=a_k$,
    \[
        0\leq y^{t,r}_s\leq \frac{a_s}{a_k}.
    \]
    Because $\bar{y}_s\in \left\{0,\frac{a_s}{a_k}\right\}$, $\bar{y}_s$ is an endpoint of this interval. The fact that $\bar{y}_s$ is a convex combination of the values $y^{t,r}_s$ therefore implies $y^{t,r}_s=\bar{y}_s$ for every $t\in[2]$ and $r\in R_t$.
    Similarly, for every $s>k$, we have $0\leq y^{t,r}_s\leq 1$ and $\bar{y}_s\in\{0,1\}$.
    Thus, $y^{t,r}_s=\bar{y}_s$ for every $t\in[2]$ and $r\in R_t$.
    Consequently, $(x^{t,r},\b{y}^{t,r}) = (\bar{x},\bar{\b{y}})$ for every $t\in[2]$ and $r\in R_t$.
    It follows that $(x^1,\b{y}^1) = (x^2,\b{y}^2) = (\bar{x},\bar{\b{y}})$.
    Therefore, $(\bar{x},\bar{\b{y}})$ is an extreme point of $\conv(C^n_i)$, and hence $D^n_{i,k}\subseteq E^n_i$.
\end{proof}

Lemmas~\ref{lem:D0} and~\ref{lem:Dk} establish that
\[
    \bigcup_{k=0}^iD^n_{i,k}\subseteq E^n_i.
\]
We next prove the reverse inclusion. For $x>0$, define the cutoff index and perturbation direction
\[
    \ell(x):=\max\bigl(\{j\in[n]\mid a_j<x\}\cup\{0\}\bigr), \qquad
    \b{d}(x,\b{y}) :=\left(x,-\sum_{j=1}^{\ell(x)}y_j\b{e}_j\right),
\]
where an empty sum is interpreted as zero.

We begin by identifying the possible values of the $x$-coordinate of a point in $E^n_i$.
\begin{lemma}\label{lem:extreme_x}
    If $(x,\b{y})\in E^n_i$, then $x\in\{0,a_1,\ldots,a_i\}$.
\end{lemma}

\begin{proof}
    Let $(x,\b{y})\in E^n_i$. Since $(x,\b{y})\in C^n_i$, we have $y_i=1$ and hence $x\leq a_i$.
    Suppose, to the contrary, that $x>0$ and $x \notin \{a_1,\ldots,a_i\}$.
    Let $\ell:=\ell(x)$. Thus $\ell<i$, $a_j<x$ for $j\leq\ell$, and $a_j>x$ for $j>\ell$. Since
    $y_j\leq a_j/x<1$ for $j\leq\ell$, we may choose $\epsilon\in(0,1)$ such that
    \[
        x(1+\epsilon)<a_{\ell+1} \quad \text{and} \quad (1+\epsilon)y_j \leq 1 \qquad (j\leq\ell).
    \]
    Set
    \[
        (x^\pm,\b{y}^\pm) := (x,\b{y})\pm\epsilon\b{d}(x,\b{y}).
    \]
    The box constraints hold, and $y_i^\pm=1$ because $i>\ell$. Moreover,
    \[
        x^\pm y_j^\pm = xy_j(1-\epsilon^2) \leq a_j \quad (j\leq\ell), \qquad
        x^\pm y_j^\pm \leq x(1+\epsilon) < a_{\ell+1}\leq a_j \quad (j>\ell).
    \]
    Hence both perturbed points belong to $C^n_i$. They are distinct and have midpoint $(x,\b{y})$, contradicting extremality.
\end{proof}

We now determine the possible $\b{y}$-coordinates. 
\begin{lemma}\label{lem:extreme_y}
    Let $(x,\b{y})\in E^n_i$.
    \begin{enumerate}
        \item If $x=0$, then $(x,\b{y})\in D^n_{i,0}$.
        \item If $x>0$, then there exists $k\in[i]$ such that $x=a_k$ and $(x,\b{y})\in D^n_{i,k}$.
    \end{enumerate}
\end{lemma}

\begin{proof}
    Since $(x,\b{y})\in C^n_i$, we have $y_i=1$.
    First suppose that $x=0$. If $y_j\in(0,1)$ for some $j\in[n]\setminus\{i\}$, then for sufficiently small $\delta>0$, both
    \[
        (0,\b{y}+\delta\b{e}_j) \quad \text{and} \quad (0,\b{y}-\delta\b{e}_j)
    \]
    belong to $C^n_i$, are distinct, and have midpoint $(0,\b{y})$. This contradicts extremality. Hence
    $y_j\in\{0,1\}$ for every $j\in[n]$, and therefore $(0,\b{y})\in D^n_{i,0}$.

    Now suppose that $x>0$. For each $j\in[n]$, define $u_j:=\min\left\{1, a_j/x \right\}$.
    Feasibility gives $0\leq y_j\leq u_j$. If $y_j\in(0,u_j)$ for some $j$, then, for sufficiently small
    $\delta>0$, both $(x,\b{y}+\delta\b{e}_j)$ and $(x,\b{y}-\delta\b{e}_j)$ belong to $C^n_i$, are distinct, and have midpoint $(x,\b{y})$, again contradicting extremality.
    Consequently, for each $j\in[n]$,
    \begin{equation}\label{eq:Ci-coordinate-endpoints}
        y_j \in 
        \begin{cases}
            \left\{0,\dfrac{a_j}{x}\right\}, & j \le \ell,\\[2mm]
            \{0,1\}, & j > \ell,
        \end{cases}
    \end{equation}
    where $\ell:=\ell(x)$. 
    
    By Lemma~\ref{lem:extreme_x}, $x=a_k$ for some $k\in[i]$. 
    We claim that $y_k=1$ for at least one such $k$. Otherwise, $y_k=0$ whenever $k\in[i]$ and $a_k=x$. Since $y_i=1$, we then have $a_i>x$, so, by monotonicity, every index $j$ satisfying
    $a_j=x$ has $j<i$ and $y_j=0$.
    As in the proof of Lemma~\ref{lem:extreme_x}, for sufficiently small $\epsilon>0$ the two points
    \[
        (x,\b{y})\pm\epsilon\b{d}(x,\b{y})
    \]
    belong to $C^n_i$: the same calculation applies for $j\leq\ell$; the coordinates with $a_j=x$ are zero; and the constraints with $a_j>x$ have strict slack. Since $\ell<i$, the equality $y_i=1$ is preserved. This contradicts extremality and proves the claim.

    Choose $k\in[i]$ with $a_k=x$ and $y_k=1$. By the endpoint condition and the ordering of $\b{a}$,
    \[
        y_j\in\left\{0,\frac{a_j}{a_k}\right\} \quad (j\leq k), \qquad
        y_j\in\{0,1\} \quad (j>k).
    \]
    Together with $y_i=y_k=1$, this yields $(x,\b{y})\in D^n_{i,k}$.
\end{proof}

\begin{proposition}\label{prop:extpts}
    For every $i\in[n]$,
    \[
        E^n_i=\bigcup_{k=0}^iD^n_{i,k}.
    \]
\end{proposition}

\begin{proof}
    Lemmas~\ref{lem:D0} and~\ref{lem:Dk} give $\bigcup_{k=0}^iD^n_{i,k}\subseteq E^n_i$. Conversely, let
    $(x,\b{y})\in E^n_i$. Lemma~\ref{lem:extreme_y} implies that $(x,\b{y})$ belongs to $D^n_{i,0}$ when $x=0$, and to $D^n_{i,k}$ for some $k\in[i]$ when $x>0$. Therefore, $E^n_i\subseteq\bigcup_{k=0}^iD^n_{i,k}$.
\end{proof}

Proposition~\ref{prop:extpts} completes the characterization of $E^n_i$. In the next subsection, we use this result to derive the vertex description of $\conv(S^n)$.

\subsection{Extreme points of $\conv(S^n)$}
For convenience, define
\[
    \mathcal U^n := \left(\bigcup_{i=0}^n C^n_i\right)\cup\{\b{0}\}, \qquad
    \mathcal V^n := \left(\bigcup_{i=0}^n E^n_i\right)\cup\{\b{0}\}
    = E^n_0 \cup \left(\bigcup_{i=1}^n \bigcup_{k=0}^i D^n_{i,k}\right) \cup \{\b{0}\}.
\]
When the underlying parameter vector is not $\b{a}$, we display it as an argument. For example, $\mathcal V^m(\b{c})$ denotes the analogous set constructed from $\b{c}$.

We first recall a standard identity for convex hulls.

\begin{lemma} \label{lem:union}
    For any family of sets $\{A_\alpha\}_{\alpha \in \mathcal{A}}$, 
    \[
    \conv \left( \bigcup_{\alpha \in \mathcal{A}} A_\alpha \right) = 
    \conv \left( \bigcup_{\alpha \in \mathcal{A}} \conv(A_\alpha) \right).
    \]
\end{lemma}

\begin{proposition}\label{prop:extpts-union}
    For every $n\geq1$,
    \[
        \conv(S^n)=\conv(\mathcal U^n)=\conv(\mathcal V^n).
    \]
\end{proposition}

\begin{proof}
    Since $\conv(C^n_i)=\conv(E^n_i)$ for every $i\in\{0,\ldots,n\}$, Lemma~\ref{lem:union} and Proposition~\ref{prop:extpts} give
    \[
        \conv(\mathcal U^n)
        =\conv\left(\left(\bigcup_{i=0}^nE^n_i\right)\cup\{\b{0}\}\right)
        =\conv(\mathcal V^n).
    \]
    Moreover, $\mathcal U^n\subseteq S^n$, and hence $\conv(\mathcal U^n)\subseteq\conv(S^n)$. It remains to prove that
    \[
        S^n\subseteq\conv(\mathcal V^n).
    \]
    We proceed by induction on $n$.

    For $n=1$, every $(x,y_1)\in S^1$ satisfies
    \[
        x+y_1\leq1+xy_1\leq1+a_1,
    \]
    where the first inequality follows from $(1-x)(1-y_1) \geq 0$. Thus, $S^1$ is contained in the polygon with vertices
    \[
        (0,0),\quad(0,1),\quad(a_1,1),\quad(1,a_1),\quad(1,0),
    \]
    which is precisely $\conv(\mathcal V^1)$; see Fig.~\ref{fig:conv(S1)}.

    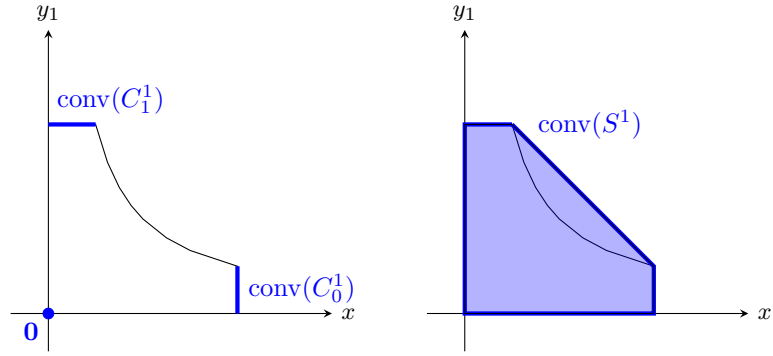
\begin{figure}[ht]
    \centering
    \begin{tikzpicture}[scale=2.5]
      \draw[-stealth] (-0.2,0) -- (1.5,0) node[right] {\small$x$};
      \draw[-stealth] (0,-0.2) -- (0,1.5) node[above] {\small$y_1$};
      \draw plot[tension=1] (0, 1) -- (1/4, 1) -- (5/16, 4/5) -- (3/8, 2/3) -- (7/16, 4/7) -- (1/2, 1/2) -- (5/8, 2/5) -- (3/4, 1/3) -- (1, 1/4) -- (1, 0);
      \draw[ultra thick, color=blue] (1, 1/4) -- (1, 0) node[above right] {$\conv(C^1_0)$};
      \draw[ultra thick, color=blue] (0, 1) -- (1/4, 1) node[above, xshift=2mm]{$\conv(C^1_1)$};
      \draw[color=blue, fill=blue] (0,0) circle (0.8pt) node[below left]{$\mathbf{0}$};
    \end{tikzpicture}
    \qquad
    \begin{tikzpicture}[scale=2.5]
      \draw[-stealth] (-0.2,0) -- (1.5,0) node[right] {\small$x$};
      \draw[-stealth] (0,-0.2) -- (0,1.5) node[above] {\small$y_1$};
      \draw plot[tension=1] (0, 1) -- (1/4, 1) -- (5/16, 4/5) -- (3/8, 2/3) -- (7/16, 4/7) -- (1/2, 1/2) -- (5/8, 2/5) -- (3/4, 1/3) -- (1, 1/4) -- (1, 0);
      \draw[ultra thick, color=blue] (1/4, 1) -- (1, 1/4) -- (1, 0) -- (0, 0) -- (0, 1) -- (1/4, 1) node[right, xshift=2mm] {$\conv(S^1)$};
      \draw[fill = blue, fill opacity = 0.3] (0, 0) -- (1, 0) -- (1, 1/4) -- (1/4, 1) -- (0, 1) -- cycle;
    \end{tikzpicture}
    \caption{Illustration of $\conv(S^1)$.}\label{fig:conv(S1)}
    \end{figure}

    Now assume, for some $\ell\geq1$, that
    \[
        S^\ell(\b{c})\subseteq\conv(\mathcal V^\ell(\b{c}))
    \]
    for every parameter vector $0<c_1\leq\cdots\leq c_\ell<1$. Consider $(x,\b{y},z)\in S^{\ell+1}(\b{a})$, where $\b{y}\in[0,1]^\ell$. Then $(x,\b{y})\in S^\ell(\b{a}_{1:\ell})$. We distinguish three cases.

    \medskip
    \noindent\textbf{Case I: $z=1$.}
    In this case,
    \[
        (x,\b{y},z)\in C^{\ell+1}_{\ell+1} \subseteq \mathcal U^{\ell+1} \subseteq \conv(\mathcal V^{\ell+1}).
    \]

    \medskip
    \noindent\textbf{Case II: $0\leq z\leq a_{\ell+1}$.}
    By the induction hypothesis, $(x,\b{y})$ is a convex combination of points $(\xi,\b{\eta}) \in \mathcal V^\ell(\b{a}_{1:\ell})$. We show that each lifted point $(\xi,\b{\eta},z)$ belongs to
    $\conv(\mathcal V^{\ell+1})$.

    If $(\xi,\b{\eta}) \in D^\ell_{i,k}(\b{a}_{1:\ell})$ for some $i \in [\ell]$ and $k \in \{0,\ldots,i\}$, then
    \[
        (\xi,\b{\eta},z) = (1-z)(\xi,\b{\eta},0)+z(\xi,\b{\eta},1),
    \]
    and both points on the right belong to $D^{\ell+1}_{i,k}$. If $(\xi,\b{\eta}) \in E^\ell_0(\b{a}_{1:\ell})$, then $\xi=1$ and
    \[
        (1,\b{\eta},z) = \left(1-\frac{z}{a_{\ell+1}}\right)(1,\b{\eta},0) + \frac{z}{a_{\ell+1}}(1,\b{\eta},a_{\ell+1}),
    \]
    where both points on the right belong to $E^{\ell+1}_0$. Finally,
    \[
        (0,\b{0},z) = (1-z)(0,\b{0},0) + z(0,\b{0},1),
    \]
    and $(0,\b{0},1)\in D^{\ell+1}_{\ell+1,0}$. Therefore, $(x,\b{y},z)\in\conv(\mathcal V^{\ell+1})$.

    \medskip
    \noindent\textbf{Case III: $a_{\ell+1}<z<1$.}
    Set $\alpha := 1+a_{\ell+1}-z$. Then $a_{\ell+1} < \alpha < 1$, and
    \[
        x \leq \frac{a_{\ell+1}}{z} \leq \frac{a_{\ell+1}}{z} - \frac{(z-a_{\ell+1})(z-1)}{z} = 1+a_{\ell+1}-z = \alpha.
    \]
    Let
    \[
        x':=\frac{x}{\alpha}, \qquad \b{a}':=\frac{\b{a}_{1:\ell}}{\alpha}.
    \]
    Since $a_\ell \leq a_{\ell+1} <\alpha$, the vector $\b{a}'$ is admissible, and $(x',\b{y})\in S^\ell(\b{a}')$. By the induction hypothesis, $(x',\b{y})$ is a convex combination of points
    $(\xi,\b{\eta})\in\mathcal V^\ell(\b{a}')$. It is therefore enough to show that
    \[
        (\alpha\xi,\b{\eta},z) \in \conv(\mathcal V^{\ell+1}) \qquad \text{for every }(\xi,\b{\eta}) \in \mathcal V^\ell(\b{a}').
    \]

    If $(\xi,\b{\eta}) \in D^\ell_{i,k}(\b{a}')$, then, directly from the definition of $D^\ell_{i,k}$,
    \[
        (\alpha\xi,\b{\eta})\in D^\ell_{i,k}(\b{a}_{1:\ell}).
    \]
    Hence, the first lifting formula from Case~II, which only requires $0 \leq z \leq 1$, applies to $(\alpha\xi,\b{\eta},z)$. The same is true when $(\xi,\b{\eta})=(0,\b{0})$.

    It remains to consider $(\xi,\b{\eta})\in E^\ell_0(\b{a}')$. Then
    \[
        (\alpha\xi,\b{\eta},z) = \left(\alpha, \frac{1}{\alpha}\b{b}\circ\b{a}_{1:\ell},z\right)
    \]
    for some $\b{b}\in\{0,1\}^\ell$. Define
    \[
        \theta:=\frac{a_{\ell+1}}{\alpha z}, \qquad
        \lambda:=\frac{1-z}{1-a_{\ell+1}}, \qquad
        \mu:=\frac{z-a_{\ell+1}}{1-a_{\ell+1}}.
    \]
    The inequality $a_{\ell+1}/z \leq \alpha$ gives $0 < \theta \leq 1$, while $\lambda, \mu \geq 0$ and $\lambda + \mu = 1$. We have
    \[
        \left(\alpha, \frac{1}{\alpha}\b{b} \circ \b{a}_{1:\ell},z \right)
        =(1-\theta)(\alpha,\b{0},z) + \theta\left(\alpha, \frac{z}{a_{\ell+1}} \b{b} \circ \b{a}_{1:\ell}, z \right).
    \]
    Moreover,
    \[
        (\alpha,\b{0},z) = \lambda(1,\b{0},a_{\ell+1})+\mu(a_{\ell+1},\b{0},1)
    \]
    and
    \[
        \left(\alpha, \frac{z}{a_{\ell+1}} \b{b} \circ \b{a}_{1:\ell}, z\right)
        =\lambda(1,\b{b}\circ\b{a}_{1:\ell},a_{\ell+1}) + \mu\left(a_{\ell+1}, \b{b} \circ \frac{\b{a}_{1:\ell}}{a_{\ell+1}}, 1 \right).
    \]
    The points with first coordinate $1$ belong to $E^{\ell+1}_0$, and the points with first coordinate $a_{\ell+1}$ belong to $D^{\ell+1}_{\ell+1,\ell+1}$. Thus, every lifted point lies in
    $\conv(\mathcal V^{\ell+1})$, completing the induction.
\end{proof}

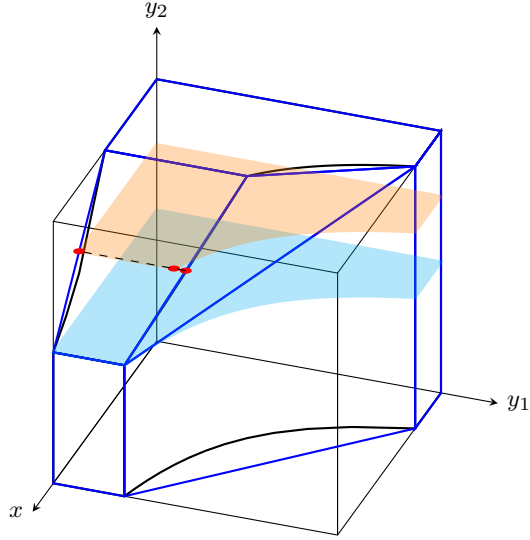
\begin{figure}[ht]
\centering
\tdplotsetmaincoords{60}{110}
\begin{tikzpicture}[scale=4,tdplot_main_coords]
\begin{scope}
  \coordinate (O) at (0,0,0);
  \draw[-stealth] (0,0,0) -- (1.2,0,0) node[left]{\small$x$};
  \draw[-stealth] (0,0,0) -- (0,1.2,0) node[right]{\small$y_1$};
  \draw[-stealth] (0,0,0) -- (0,0,1.2) node[above]{\small$y_2$};
  \draw (0, 0, 0) -- (1, 0, 0) -- (1, 0, 1) -- (0, 0, 1);
  \draw (0, 0, 1) -- (0, 1, 1) -- (0, 1, 0) -- (1, 1, 0) -- (1, 0, 0);
  \draw (1, 0, 1) -- (1, 1, 1) -- (1, 1, 0) -- (1, 1, 1) -- (0, 1, 1);
  \draw[thick] (1, 1/4, 0) -- (1, 0, 0) -- (1, 0, 1/2);
  \draw[thick] (0, 0, 1) -- (1/2, 0, 1);
  \draw[thick] (0, 0, 1) -- (0, 1, 1) -- (1/4, 1, 1) -- (0, 1, 1) -- (0, 1, 0) -- (1/4, 1, 0);
  \draw[thick] plot[tension=1] (1, 0, 1/2) -- (7/8, 0, 4/7) -- (3/4, 0, 2/3) -- (1/2, 0, 1) -- (1/2, 1/2, 1) -- (0.5,  0.5, 1) -- (0.5333, 0.468779, 0.937559) -- (0.5667, 0.441151, 0.882301) -- (0.6, 0.416667, 0.833333) -- (0.6333, 0.394758, 0.789515) -- (0.6667, 0.374981, 0.749963) -- (0.7, 0.357143, 0.714286) -- (0.7333, 0.340925, 0.681849) -- (0.7667, 0.326073, 0.652146) -- (0.8, 0.3125, 0.625) -- (0.8333, 0.300012, 0.600024) -- (0.8667, 0.28845, 0.576901) -- (0.9, 0.277778, 0.555556) -- (0.9333, 0.267867, 0.535733) -- (0.9667, 0.258612, 0.517224) -- (1, 0.25, 0.5) -- (1, 0, 0.5) -- (1, 0.25, 0.5) -- (1, 0.25, 0);
  \draw[thick] plot[tension=1] (1, 0.25, 0) -- (0.9375,0.266666667, 0) -- (0.8759, 0.28542071, 0) -- (0.8125, 0.307692308, 0) -- (0.75,	0.333333333, 0) -- (0.6875, 0.363636364, 0) -- (0.625, 0.4, 0) -- (0.5625, 0.444444444, 0) -- (0.5, 0.5, 0) -- (0.4375, 0.571428571, 0) -- (0.375, 0.666666667, 0) -- (0.3125, 0.8, 0)--(0.25, 1, 0) -- (0.25, 1, 1);
  \draw[thick] plot[tension=1] (0.5, 0.5, 1)--(0.4936, 0.506482982, 1)--(0.4872, 0.513136289, 1)--(0.4808, 0.519966722, 1)--(0.4744, 0.52698145, 1)--(0.4679, 0.534302201, 1)--(0.4615, 0.541711809, 1)--(0.4551, 0.549329818, 1)--(0.4487, 0.557165144, 1)--(0.4423, 0.565227221, 1)--(0.4359, 0.573526038, 1)--(0.4295, 0.582072177, 1)--(0.4231, 0.590876861, 1)--(0.4167, 0.599952004, 1)--(0.4103, 0.609310261, 1)--(0.4038, 0.619118375, 1)--(0.3974, 0.629089079, 1)--(0.391, 0.639386189, 1)--(0.3846, 0.650026001, 1)--(0.3782, 0.661025912, 1)--(0.3718, 0.672404519, 1)--(0.3654, 0.684181719, 1)--(0.359, 0.69637883, 1)--(0.3526, 0.709018718, 1)--(0.3462, 0.722125939, 1)--(0.3397, 0.73594348, 1)--(0.3333, 0.750075008, 1)--(0.3269, 0.764759865, 1)--(0.3205, 0.780031201, 1)--(0.3141, 0.795924865, 1)--(0.3077, 0.812479688, 1)--(0.3013, 0.829737803, 1)--(0.2949, 0.847744998, 1)--(0.2885, 0.866551127, 1)--(0.2821, 0.886210564, 1)--(0.2756, 0.907111756, 1)--(0.2692, 0.928677563, 1)--(0.2628, 0.95129376, 1)--(0.2564, 0.975039002, 1)--(0.25, 1, 1);
  \draw[thick, color=blue] (1/2, 0, 1) -- (1/2, 1/2, 1) -- (1, 1/4, 1/2) -- (1, 0, 1/2) -- (1/2, 0, 1);
  \draw[thick, color=blue] (1, 1/4, 1/2) -- (1, 1/4, 0) -- (1/4, 1, 0) -- (1/4, 1, 1) -- (1, 1/4, 1/2);
  \draw[thick, color=blue] (1, 1/4, 1/2) -- (1/4, 1, 1) -- (1/2, 1/2, 1) -- (1, 1/4, 1/2);
  \draw[thick, color=blue] (0, 0, 1) -- (0, 1, 1) -- (1/4, 1, 1) -- (1/2, 1/2, 1) -- (1/2, 0, 1) -- (0, 0, 1);
  \draw[thick, color=blue] (1, 0, 1/2) -- (1, 1/4, 1/2) -- (1, 1/4, 0) -- (1, 0, 0) -- (1, 0, 1/2);
  \draw[thick, color=blue] (1/4, 1, 1) -- (0, 1, 1) -- (0, 1, 0) -- (1/4, 1, 0) -- (1/4, 1, 1);
  \filldraw[very thick, color=cyan, opacity=0.3] (0, 0, 1/2) -- (0, 1, 1/2) -- (1/4, 1, 1/2) -- (0.265306122, 0.942307692, 0.50) -- (0.280612245, 0.890909091, 0.50) -- (0.295918367, 0.844827586, 0.50) -- (0.31122449, 0.803278689, 0.50) -- (0.326530612, 0.765625, 0.50) -- (0.341836735, 0.731343284, 0.50) -- (0.357142857, 0.7, 0.50) -- (0.37244898, 0.671232877, 0.50) -- (0.387755102, 0.644736842, 0.50) -- (0.403061224, 0.620253165, 0.50) -- (0.418367347, 0.597560976, 0.50) -- (0.433673469, 0.576470588, 0.50) -- (0.448979592, 0.556818182, 0.50) -- (0.464285714, 0.538461538, 0.50) -- (0.479591837, 0.521276596, 0.50) -- (0.494897959, 0.505154639, 0.50) -- (0.510204082, 0.49, 0.50) -- (0.525510204, 0.475728155, 0.50) -- (0.540816327, 0.462264151, 0.50) -- (0.556122449, 0.449541284, 0.50) -- (0.571428571, 0.4375, 0.50) -- (0.586734694, 0.426086957, 0.50) -- (0.602040816, 0.415254237, 0.50) -- (0.617346939, 0.404958678, 0.50) -- (0.632653061, 0.39516129, 0.50) -- (0.647959184, 0.385826772, 0.50) -- (0.663265306, 0.376923077, 0.50) -- (0.678571429, 0.368421053, 0.50) -- (0.693877551, 0.360294118, 0.50) -- (0.709183673, 0.352517986, 0.50) -- (0.724489796, 0.345070423, 0.50) -- (0.739795918, 0.337931034, 0.50) -- (0.755102041, 0.331081081, 0.50) -- (0.770408163, 0.324503311, 0.50) -- (0.785714286, 0.318181818, 0.50) -- (0.801020408, 0.312101911, 0.50) -- (0.816326531, 0.30625, 0.50) -- (0.831632653, 0.300613497, 0.50) -- (0.846938776, 0.295180723, 0.50) -- (0.862244898, 0.289940828, 0.50) -- (0.87755102, 0.284883721, 0.50) -- (0.892857143, 0.28, 0.50) -- (0.908163265, 0.275280899, 0.50) -- (0.923469388, 0.270718232, 0.50) -- (0.93877551, 0.266304348, 0.50) -- (0.954081633, 0.262032086, 0.50) -- (0.969387755, 0.257894737, 0.50) -- (0.984693878, 0.25388601, 0.50) -- (1, 1/4, 1/2) -- (1, 0, 1/2) -- (0, 0, 1/2);
  \filldraw[very thick, color=orange, opacity=0.3] (0, 0, 3/4) -- (0.000 , 1.000 , 3/4) -- (1/4, 1, 3/4) -- 
( 0.260204082 , 0.960784314 , 0.75 ) --
( 0.270408163 , 0.924528302 , 0.75 ) --
( 0.280612245 , 0.890909091 , 0.75 ) --
( 0.290816327 , 0.859649123 , 0.75 ) --
( 0.301020408 , 0.830508475 , 0.75 ) --
( 0.31122449 , 0.803278689 , 0.75 ) --
( 0.321428571 , 0.777777778 , 0.75 ) --
( 0.331632653 , 0.753846154 , 0.75 ) --
( 0.341836735 , 0.731343284 , 0.75 ) --
( 0.352040816 , 0.710144928 , 0.75 ) --
( 0.362244898 , 0.690140845 , 0.75 ) --
( 0.37244898 , 0.671232877 , 0.75 ) --
( 0.382653061 , 0.653333333 , 0.75 ) --
( 0.392857143 , 0.636363636 , 0.75 ) --
( 0.403061224 , 0.620253165 , 0.75 ) --
( 0.413265306 , 0.604938272 , 0.75 ) --
( 0.423469388 , 0.590361446 , 0.75 ) --
( 0.433673469 , 0.576470588 , 0.75 ) --
( 0.443877551 , 0.563218391 , 0.75 ) --
( 0.454081633 , 0.550561798 , 0.75 ) --
( 0.464285714 , 0.538461538 , 0.75 ) --
( 0.474489796 , 0.52688172 , 0.75 ) --
( 0.484693878 , 0.515789474 , 0.75 ) --
( 0.494897959 , 0.505154639 , 0.75 ) --
( 0.505102041 , 0.494949495 , 0.75 ) --
( 0.515306122 , 0.485148515 , 0.75 ) --
( 0.525510204 , 0.475728155 , 0.75 ) --
( 0.535714286 , 0.466666667 , 0.75 ) --
( 0.545918367 , 0.457943925 , 0.75 ) --
( 0.556122449 , 0.449541284 , 0.75 ) --
( 0.566326531 , 0.441441441 , 0.75 ) --
( 0.576530612 , 0.433628319 , 0.75 ) --
( 0.586734694 , 0.426086957 , 0.75 ) --
( 0.596938776 , 0.418803419 , 0.75 ) --
( 0.607142857 , 0.411764706 , 0.75 ) --
( 0.617346939 , 0.404958678 , 0.75 ) --
( 0.62755102 , 0.398373984 , 0.75 ) --
( 0.637755102 , 0.392 , 0.75 ) --
( 0.647959184 , 0.385826772 , 0.75 ) --
( 0.658163265 , 0.379844961 , 0.75 ) --
( 0.668367347 , 0.374045802 , 0.75 ) --
( 0.678571429 , 0.368421053 , 0.75 ) --
( 0.68877551 , 0.362962963 , 0.75 ) --
( 0.698979592 , 0.357664234 , 0.75 ) --
( 0.709183673 , 0.352517986 , 0.75 ) --
( 0.719387755 , 0.34751773 , 0.75 ) --
( 0.729591837 , 0.342657343 , 0.75 ) --
( 0.739795918 , 0.337931034 , 0.75 ) --
( 0.75 , 0.333333333 , 0.75 ) -- (0.75, 0, 0.75) -- (0, 0, 0.75);
\fill[red] ( 0.75 , 0.333333333 , 0.75 ) circle(.02);
  \fill[red] (0.75, 0, 0.75) circle(.02);
  \fill[red] (0.75, 3/8, 0.75) circle(.02);
  \draw[dashed] (0.75, 3/8, 0.75) -- ( 0.75 , 0.333333333 , 0.75 ) -- (0.75, 0, 0.75);
  \end{scope}
  \end{tikzpicture}
  \caption{Illustration of Case~III in the proof of Proposition~\ref{prop:extpts-union} for $n=2$.}\label{fig:caseIII}
  \end{figure}
  
In Fig.~\ref{fig:caseIII}, we illustrate the geometric idea of Case III in the proof of Proposition~\ref{prop:extpts-union} for \(n=2\). The cyan slice corresponds to \(y_2=a_2\), which is the threshold separating Cases II and III. The orange slice corresponds to a representative value \(z\in(a_2,1)\) and depicts the set
\[
    \{(x,y_1,z)\mid (x/\alpha,y_1)\in S^1(a_1/\alpha)\}.
\]
The three red points demonstrate the first convex combination in the $E^\ell_0$ case: the point
$(\alpha,a_1/\alpha,z)$ lies on the segment joining $(\alpha,0,z)$ and $(\alpha,a_1z/a_2,z)$.

We summarize the conclusion of this section in the following theorem.

\begin{theorem}\label{thm:extpts}
    The set of extreme points of $\conv(S^n)$ is
    \begin{align*}
        \mathcal V^n &= \left\{(0,\b{b}) \, \middle| \, \b{b}\in\{0,1\}^n\right\}
        \cup \left(\bigcup_{k=1}^n D^n_{k,k}\right) \cup E^n_0 \\
        &= \left\{(0,\b{b}) \, \middle| \, \b{b} \in \{0,1\}^n \right\} 
        \cup \left( \bigcup_{k=1}^n \left \{ \left(a_k,\b{b} \circ \b{\gamma}^{\,k}\right) \, \middle| \, \b{b} \in \{0,1\}^n, \ b_k=1 \right\} \right) \\
        & \qquad \cup \left \{(1, \b{b} \circ \b{a}) \, \middle| \, \b{b} \in \{0,1\}^n \right\}.
    \end{align*}
    If $a_1,\ldots,a_n$ are pairwise distinct, then $\conv(S^n)$ has $(n+4)2^{n-1}$ extreme points.
\end{theorem}

\begin{proof}
    We first verify that the displayed sets agree with the earlier definition 
    \[
        \mathcal V^n = E^n_0 \cup \left(\bigcup_{i=1}^n\bigcup_{k=0}^iD^n_{i,k} \right) \cup\{\b{0}\}.
    \]
    Since every nonzero binary vector has at least one component equal to one,
    \[
        \{\b{0}\} \cup \left(\bigcup_{i=1}^n D^n_{i,0} \right)
        =\left\{(0,\b{b}) \, \middle| \, \b{b}\in\{0,1\}^n\right\}.
    \]
    Moreover, if $1\leq k\leq i\leq n$, then $D^n_{i,k}\subseteq D^n_{k,k}$, while $D^n_{k,k}$ itself occurs in the union when $i=k$. Hence
    \[
        \bigcup_{i=1}^n\bigcup_{k=1}^i D^n_{i,k}
        =\bigcup_{k=1}^n D^n_{k,k}.
    \]
    These identities give the first displayed representation of $\mathcal V^n$. The second follows directly from the definitions of $D^n_{k,k}$ and $E^n_0$.

    Proposition~\ref{prop:extpts-union} gives $\conv(S^n)=\conv(\mathcal V^n)$. Since $\mathcal V^n$ is
    finite, every extreme point of $\conv(S^n)$ belongs to $\mathcal V^n$.
    For the reverse inclusion, observe that
    \[
        \conv(C^n_0)=\conv(S^n)\cap\{(x,\b{y}) \mid x=1 \}
    \]
    and, for each $i\in[n]$,
    \[
        \conv(C^n_i)=\conv(S^n) \cap \{(x,\b{y}) \mid y_i=1\}.
    \]
    Indeed, if a convex combination of points in $S^n$ has $x=1$ or $y_i=1$, then every positive-weight point in that combination has the same coordinate equal to $1$. Hence, for each $i \in \{0, \ldots, n\}$, $\conv(C^n_i)$ is an exposed face of $\conv(S^n)$, and every point of $E^n_i$ is an extreme point of $\conv(S^n)$. The origin is also an extreme point of $\conv(S^n)$, because $\conv(S^n) \subseteq \Re^{n+1}_+$. Recalling that
    \[
        \mathcal V^n = \{\b{0}\} \cup \left(\bigcup_{i=0}^nE^n_i \right),
    \]
    we conclude that every point of $\mathcal V^n$ is an extreme point of $\conv(S^n)$.

    Finally, suppose that $a_1,\ldots,a_n$ are pairwise distinct. Since $0<a_k<1$ for every $k \in [n]$, the $n+2$ first-coordinate values $0, a_1, \ldots, a_n, 1$ are pairwise distinct. Hence the corresponding classes in the displayed description of $\mathcal V^n$ are pairwise disjoint. The classes with first coordinates $0$ and $1$ each contain $2^n$ points, while, for each $k\in[n]$, the class with first coordinate $a_k$ contains $2^{n-1}$ points. Therefore,
    \[
        \left|\mathcal V^n\right| = 2^n+n2^{n-1}+2^n = (n+4)2^{n-1}.
    \]
\end{proof}

\section{The facet description of $\conv(S^n)$}\label{sec:facets}

In this section, we derive a complete facet description of $\conv(S^n)$. We first establish a lifting result that transfers valid and facet-defining inequalities between $\conv(S^n)$ and the lower-dimensional sets obtained by restricting the index set. We then introduce the \emph{leap inequalities} and prove that they define facets. Finally, we show that every facet of $\conv(S^n)$ is either a variable bound or a leap inequality.

\subsection{A lifting theorem}
Let $Q := \{i_1, \cdots, i_q\}$ be a nonempty subset of $[n]$ with $i_1 < \cdots < i_q$. For $\b{y} \in \Re^n$, define
\[
    \b{y}^Q := (y_{i_1},\ldots,y_{i_q}) \in \Re^q.
\]
For $\b{u} \in \Re^q$, let $\iota_Q(\b{u}) \in \Re^n$ denote its zero extension to the coordinates in $Q$, that is,
\[
    \bigl( \iota_Q(\b{u}) \bigr)_j :=
    \begin{cases}
        u_k, & j=i_k\text{ for some }k\in[q],\\
        0,   & j \notin Q.
    \end{cases}
\]
Thus, for every $\b{c}\in\Re^q$ and $\b{y}\in\Re^n$,
\begin{equation}\label{eqn:cyQ}
    \iota_Q(\b{c})^\intercal\b{y}
    =\b{c}^\intercal\b{y}^Q.
\end{equation}

The next lemma records the required projection and extension properties.

\begin{lemma}\label{lem:yQ}
    Let $Q := \{i_1, \cdots, i_q\}$ be a nonempty subset of $[n]$ with $i_1 < \cdots < i_q$. 
    \begin{enumerate}
        \item If $(x,\b{y}) \in S^n$, then $(x,\b{y}^Q) \in S^q(\b{a}^Q)$. Conversely, if $(x,\b{u}) \in S^q(\b{a}^Q)$, then $(x,\iota_Q(\b{u})) \in S^n$ and, for every $j \notin Q$,
        \[
            (x,\iota_Q(\b{u})+a_j\b{e}_j)\in S^n.
        \]
        \item The same statements hold with $S^n$ and $S^q(\b{a}^Q)$ replaced by their convex hulls. In particular,
        \[
            \left\{(x,\b{y}^Q) \, \middle| \, (x,\b{y}) \in \conv(S^n)\right\} = \conv(S^q(\b{a}^Q)).
        \]
    \end{enumerate}
\end{lemma}

\begin{proof}
    Part~1 follows directly from the defining inequalities. In particular, if $(x,\b{u})\in S^q(\b{a}^Q)$ and $j\notin Q$, then $x a_j\leq a_j$ because $0 \leq x \leq 1$.

    For Part~2, projection preserves convex combinations. Hence, the projection of $\conv(S^n)$ is contained in $\conv(S^q(\b{a}^Q))$. Conversely, write $(x,\b{u})\in\conv(S^q(\b{a}^Q))$ as a finite convex combination of points in $S^q(\b{a}^Q)$ and apply Part~1 term by term. The same convex combination gives $(x,\iota_Q(\b{u}))\in\conv(S^n)$ and, for every $j\notin Q$, $(x,\iota_Q(\b{u})+a_j\b{e}_j) \in \conv(S^n)$.
\end{proof}

We can now state the lifting result.

\begin{theorem}\label{thm:lifting}
    Let $Q := \{i_1, \cdots, i_q\}$ be a nonempty subset of $[n]$ with $i_1 < \cdots < i_q$, and let $c_0, d \in \Re$ and $\b{c} \in \Re^q$. Consider the inequalities
    \begin{align}
        c_0x+\iota_Q(\b{c})^\intercal\b{y} & \leq d, \label{eqn:lift_high} \\
        c_0x+\b{c}^\intercal\b{u} & \leq d. \label{eqn:lift_low}
    \end{align}
    Then:
    \begin{enumerate}
        \item Inequality \eqref{eqn:lift_high} is valid for $\conv(S^n)$ if and only if \eqref{eqn:lift_low} is valid for $\conv(S^q(\b{a}^Q))$.
        \item Inequality \eqref{eqn:lift_high} defines a facet for $\conv(S^n)$ if and only if \eqref{eqn:lift_low} defines a facet for $\conv(S^q(\b{a}^Q))$.
    \end{enumerate}
\end{theorem}

\begin{proof}
    The validity equivalence follows immediately from Lemma~\ref{lem:yQ} and identity~\eqref{eqn:cyQ}.

    Let $P_n:=\conv(S^n)$ and $P_Q:=\conv(S^q(\b{a}^Q))$, and let $F_n$ and $F_Q$ denote the faces on which \eqref{eqn:lift_high} and \eqref{eqn:lift_low}, respectively, hold at equality.
    Suppose first that \eqref{eqn:lift_high} defines a facet of $P_n$. The projection
    \[
        \pi_Q(x,\b{y}):=(x,\b{y}^Q)
    \]
    maps $F_n$ onto $F_Q$. Since $\dim(F_n)=n$ and $\ker(\pi_Q)$ has dimension $n-q$,
    \[
        \dim(F_Q)\geq\dim(\pi_Q(F_n))\geq n-(n-q)=q.
    \]
    As $F_Q$ is a proper face of the full-dimensional polytope $P_Q\subseteq\Re^{q+1}$, its dimension is at most $q$. Hence $\dim(F_Q)=q$, and \eqref{eqn:lift_low} defines a facet of $P_Q$.

    Conversely, suppose that \eqref{eqn:lift_low} defines a facet of $P_Q$. Since $P_Q$ is full-dimensional in $\Re^{q+1}$, there exist $q+1$ affinely independent points
    \[
        \b{p}^r=(x^r,\b{u}^r)\in F_Q, \qquad r=0,\ldots,q.
    \]
    By Lemma \ref{lem:yQ} and equation \eqref{eqn:cyQ}, their zero extensions
    \[
    \widehat{\b{p}}^{\,r} := \left(x^r,\iota_Q(\b{u}^r)\right), \qquad r=0,\ldots,q,
    \]
    belong to $F_n$. Moreover, the $n-q$ points
    \[
        \b{r}^j := \widehat{\b{p}}^{\,0} + (0,a_j\b{e}_j), \qquad j \in [n]\setminus Q
    \]
    also belong to $F_n$ by Lemma \ref{lem:yQ} and the fact that \eqref{eqn:lift_high} has zero coefficients on the coordinates outside $Q$.

    Thus, $F_n$ contains the $n+1$ points
    \[
        \widehat{\b{p}}^{\,0}, \ldots, \widehat{\b{p}}^{\,q}, \quad \b{r}^j\ \ (j\in[n]\setminus Q),
    \]
    which are affinely independent. Indeed, relative to $\widehat{\b{p}}^{\,0}$, the first $q$ differences are linearly independent and supported only on the $x$-coordinate and the coordinates in $Q$, while the remaining $n-q$ differences are $(0,a_j\b{e}_j)$, $j \notin Q$, and are supported on distinct coordinates outside $Q$. Therefore, $F_n$ is an $n$-dimensional proper face of the full-dimensional polytope $P_n \subseteq \Re^{n+1}$, and it follows that \eqref{eqn:lift_high} defines a facet of $P_n$.
\end{proof}

For $n=1$, the vertex description in Section~\ref{sec:extpts} gives
\[
    \conv(S^1(a)) = \left\{(x,y)\in\Re^2 \, \middle| \, 0 \leq x \leq1, 0 \leq y \leq1, x+y \leq 1+a
    \right\}.
\]
Each of these five inequalities defines a facet. Applying Theorem~\ref{thm:lifting} with $Q=\{i\}$ yields the following corollary.

\begin{corollary}\label{cor:lifting}
    Each inequality in the following systems is facet-defining for $\conv(S^n)$:
    \begin{align}
        &0\leq x\leq1, \label{eqn:bdx}\\
        &0\leq y_i\leq1, &&i\in[n], \label{eqn:bdy}\\
        &x+y_i\leq1+a_i, &&i\in[n]. \label{eqn:Mc2}
    \end{align}
\end{corollary}

We refer to \eqref{eqn:bdx} and \eqref{eqn:bdy} as the \emph{bound inequalities}. Inequality~\eqref{eqn:Mc2} is the projected McCormick inequality for the bilinear constraint $xy_i\leq a_i$. In general, these one-coordinate inequalities are insufficient: whenever $\b{a}$ has at least two distinct components, additional facets are needed to capture the coupling induced by the common variable $x$. We introduce those facets next.

\subsection{The leap inequalities} \label{sec:leap}
Let $Q = \{i_1, \cdots, i_q\}$ be a nonempty subset of $[n]$ such that $a_{i_1} < \cdots < a_{i_q}$.
Set $a_{i_0}:=0$ and $a_{i_{q+1}}:=1$. The \emph{leap inequality} associated with $Q$ is
\begin{equation}\label{eqn:leap}
    x+\sum_{k=1}^{q} \bigl(a_{i_{k+1}}-a_{i_{k-1}}\bigr)y_{i_k} \leq 1+a_{i_q}.
\end{equation}
When $Q=\{i\}$, inequality~\eqref{eqn:leap} reduces to \eqref{eqn:Mc2}.

To prove validity and facetness, define
\begin{equation}\label{eqn:zk}
    \b{z}^i:=(a_i,\b{\gamma}^{\,i})
    \quad(i\in[n]),
    \qquad
    \b{z}^{n+1}:=(1,\b{a}),
\end{equation}
where $\b{\gamma}^{\,i}$ was introduced in Section~\ref{sec:extpts}. These are the componentwise maximal vertices corresponding to the possible positive values of the $x$-coordinate.

\begin{lemma}\label{lem:zk}
    The following properties hold.
    \begin{enumerate}
        \item For every extreme point $\b{z}$ of $\conv(S^n)$, there exists $i\in[n+1]$ such that $\b{0} \leq \b{z} \leq \b{z}^i$. \label{lem:zk-1}
        \item The points $\b{z}^1,\ldots,\b{z}^{n+1}$ are affinely independent if and only if the components of $\b{a}$ are pairwise distinct. \label{lem:zk-2}
        \item If $a_1 < \cdots < a_n$, then the leap inequality associated with $Q=[n]$ is active at every point $\b{z}^i$, $i\in[n+1]$. \label{lem:zk-3}
    \end{enumerate}
\end{lemma}

\begin{proof}
\begin{enumerate}
    \item The claim follows directly from Theorem~\ref{thm:extpts}. A vertex with first coordinate $a_i$ is dominated by $\b{z}^i$, a vertex with first coordinate $1$ is dominated by $\b{z}^{n+1}$, and a vertex with first coordinate $0$ is dominated by $\b{z}^1$.

    \item Let $\Omega \in \Re^{(n+2)\times(n+1)}$ be the matrix whose $i$th column is $\begin{bmatrix}
    1 \\ \b{z}^i \end{bmatrix}$ for $i \in [n+1]$. Thus,
    \[
        \Omega = \begin{bmatrix} 1 & \cdots & 1 \\ \b{z}^1 & \cdots & \b{z}^{n+1} \end{bmatrix}
        = \begin{bmatrix}
            1   & 1         & 1         & \cdots & 1        & 1 \\
            a_1 & a_2       & a_3       & \cdots & a_n      & 1 \\
            1   & a_1/a_2   & a_1/a_3   & \cdots & a_1/a_n  & a_1 \\
            1   & 1         & a_2/a_3   & \cdots & a_2/a_n  & a_2 \\
            \vdots & \vdots & \vdots    & \ddots & \vdots   & \vdots \\
            1   & 1         & 1         & \cdots & a_{n-1}/a_n  & a_{n-1} \\
            1   & 1         & 1         & \cdots & 1        & a_n
        \end{bmatrix}.
    \]
    The points $\b{z}^1,\ldots,\b{z}^{n+1}$ are affinely independent if and only if the columns of $\Omega$ are linearly independent, or equivalently, if and only if
    \[
        \rank(\Omega)=n+1.
    \]

    Subtract the first row of $\Omega$ from each of its remaining rows. This row operation preserves the rank and produces
    \[
        \Omega' =
        \begin{bmatrix}
            1     & 1           & 1           & \cdots  & 1             & 1 \\
            a_1-1 & a_2-1       & a_3-1       & \cdots  & a_n-1         & 0 \\
            0     & a_1/a_2-1   & a_1/a_3-1   & \cdots  & a_1/a_n-1     & a_1-1 \\
            0     & 0           & a_2/a_3-1   & \cdots  & a_2/a_n-1     & a_2-1 \\
            \vdots& \vdots      & \vdots      & \ddots  & \vdots        & \vdots \\
            0     & 0           & 0           & \cdots  & a_{n-1}/a_n-1 & a_{n-1}-1 \\
            0     & 0           & 0           & \cdots  & 0             & a_n-1
        \end{bmatrix}.
    \]

    Suppose first that the components of $\b{a}$ are distinct. Delete the first row of $\Omega'$. The resulting $(n+1)\times(n+1)$ submatrix is upper triangular, with diagonal entries
    \[
        a_1-1, \quad \frac{a_1}{a_2}-1, \quad \frac{a_2}{a_3}-1, \quad \ldots, \quad \frac{a_{n-1}}{a_n}-1, \quad a_n-1.
    \]
    Each of these entries is nonzero. Therefore, this submatrix has a nonzero determinant, and hence
    \[
        \rank(\Omega) = \rank(\Omega') = n+1.
    \]
    It follows that $\b{z}^1,\ldots,\b{z}^{n+1}$ are affinely independent.

    Conversely, suppose that the components of $\b{a}$ are not distinct. Since $a_1\leq\cdots\leq a_n$, there exists $i\in[n-1]$ such that $a_i=a_{i+1}$. In this case, the $i$th and $(i+1)$st columns of $\Omega$ are identical. Consequently, $\rank(\Omega)<n+1$, so $\b{z}^1, \ldots, \b{z}^{n+1}$ are affinely dependent.

    \item Let $\Lambda$ denote the left-hand side of the leap inequality for $Q=[n]$, with $a_0=0$ and $a_{n+1}=1$. For $i\in[n]$,
    \begin{align*}
        \Lambda(\b{z}^i)
        &= a_i + \frac{1}{a_i}\sum_{k=1}^{i} a_k(a_{k+1}-a_{k-1}) + \sum_{k=i+1}^{n}(a_{k+1}-a_{k-1}) \\
        &= a_i + a_{i+1} + \bigl(1+a_n-a_i-a_{i+1}\bigr) = 1+a_n.
    \end{align*}
    Similarly,
    \[
        \Lambda(\b{z}^{n+1}) = 1+\sum_{k=1}^{n}a_k(a_{k+1}-a_{k-1}) = 1+a_n.
    \]
\end{enumerate}
\end{proof}

\begin{proposition}\label{prop:leap}
    For every nonempty set $Q = \{i_1, \cdots, i_q\} \subseteq [n]$ such that $a_{i_1} < \cdots < a_{i_q}$, the leap inequality~\eqref{eqn:leap} defines a facet of $\conv(S^n)$.
\end{proposition}

\begin{proof}
    First suppose that $Q=[n]$ and $a_1<\cdots<a_n$. Let $\Lambda$ denote the left-hand side of \eqref{eqn:leap}. For every extreme point $\b{z}$ of $\conv(S^n)$, Part~\ref{lem:zk-1} of Lemma~\ref{lem:zk} gives an index $i\in[n+1]$ such that $\b{0} \leq \b{z} \leq \b{z}^i$. All coefficients of $\Lambda$ are nonnegative, and Part~\ref{lem:zk-3} gives $\Lambda(\b{z}^i)=1+a_n$. Hence $\Lambda(\b{z})\leq1+a_n$.
    The inequality is therefore valid for every extreme point, and hence for $\conv(S^n)$. By Parts~\ref{lem:zk-2} and~\ref{lem:zk-3}, the $n+1$ affinely independent points $\b{z}^1,\ldots,\b{z}^{n+1}$ satisfy the inequality at equality. Thus, it defines a facet.

    For a general set $Q$, the preceding argument applied to $\conv(S^q(\b{a}^Q))$ shows that the full-support leap inequality in the reduced space is facet-defining. Theorem~\ref{thm:lifting} then lifts it to \eqref{eqn:leap} in the original space.
\end{proof}

\subsection{Completeness of the facet description}

The bound and leap inequalities are all facet-defining. We now prove that no other facets occur. We begin with the monotonicity of $S^n$ and its convex hull.

\begin{lemma}\label{lem:inclusionxy}
    Let $0\leq u\leq x$ and $\b{0}\leq\b{v}\leq\b{y}$.
    \begin{enumerate}
        \item If $(x,\b{y})\in S^n$, then $(u,\b{v})\in S^n$.
        \label{lem:inclusionxy-1}
        \item If $(x,\b{y})\in\conv(S^n)$, then $(u,\b{v})\in\conv(S^n)$.
        \label{lem:inclusionxy-2}
    \end{enumerate}
\end{lemma}

\begin{proof}
    Part~\ref{lem:inclusionxy-1} follows immediately from the definition of $S^n$.

    For Part~\ref{lem:inclusionxy-2}, choose $\alpha_0,\ldots,\alpha_n\in[0,1]$ such that
    \[
        u=\alpha_0x
        \qquad\text{and}\qquad
        v_j=\alpha_jy_j\quad(j\in[n]).
    \]
    Write
    \[
        (x,\b{y})=\sum_{t=1}^{m}\lambda_t(x^t,\b{y}^t),
    \]
    where $\lambda_t\geq0$, $\sum_t\lambda_t=1$, and $(x^t,\b{y}^t)\in S^n$. Then
    \[
        (u,\b{v})
        =\sum_{t=1}^{m}\lambda_t
        \bigl(\alpha_0x^t,
        \alpha_1y^t_1,\ldots,\alpha_ny^t_n\bigr).
    \]
    Each point in this convex combination belongs to $S^n$ by Part~\ref{lem:inclusionxy-1}, proving the claim.
\end{proof}

We next classify facets with a negative coefficient and facets that do not involve $x$.

\begin{proposition}\label{prop:negprop}
    Let $c_0x + \b{c}^\intercal\b{y} \leq d$ define a facet of $\conv(S^n)$. If $c_i<0$ for some $i\in\{0\} \cup [n]$, where index $0$ corresponds to $x$, then $d=0$ and $c_j=0$ for every $j\neq i$. Consequently, the inequality is a positive scalar multiple of $x \geq 0$ when $i=0$, or of $y_i \geq 0$ when $i \in [n]$.
\end{proposition}

\begin{proof}
    Let 
    \[
        F := \conv(S^n) \cap \{(x,\b{y}) \mid c_0 x + \b{c}^\intercal \b{y} = d\}
    \]
    be the facet defined by the inequality $c_0 x + \b{c}^\intercal \b{y} \le d$.
    First, consider the case where $c_i < 0$ for some $i \in [n]$. For an arbitrary point $(\bar{x}, \bar{\b{y}}) \in F$, we claim that $\bar{y}_i = 0$. Otherwise, suppose that $\bar{y}_i > 0$. By Lemma~\ref{lem:inclusionxy}, we have $(\bar{x}, \bar{\b{y}} - \epsilon \b{e}_i) \in \conv(S^n)$
    for a sufficiently small $\epsilon > 0$. For such an $\epsilon$, we obtain
    \[
        c_0 \bar{x} + \b{c}^\intercal (\bar{\b{y}} - \epsilon \b{e}_i)
        = c_0 \bar{x} + \b{c}^\intercal \bar{\b{y}} - \epsilon c_i
        > c_0 \bar{x} + \b{c}^\intercal \bar{\b{y}}
        = d.
    \]
    This contradicts the validity of the inequality $c_0 x + \b{c}^\intercal\b{y} \le d$. Hence, every point in $F$ satisfies $\bar{y}_i = 0$. As a result, the facet $F$ is a subset of $\conv(S^n) \cap \{(x,\b{y}) \mid y_i = 0\}$, which, by Corollary~\ref{cor:lifting}, is itself a facet of $\conv(S^n)$. Since two facets of a full-dimensional polytope cannot be properly nested, the two faces are equal. Their supporting hyperplanes therefore coincide, and the corresponding facet-defining inequalities agree up to positive scaling. This implies that $c_j = 0$ for all $j\in\bigl(\{0\}\cup[n]\bigr)\setminus\{i\}$ and $d = 0$.
    
    The case where $c_0 < 0$ is handled analogously by applying the same argument to the $x$-coordinate.
\end{proof}

\begin{proposition}\label{prop:poscofacets-x}
    Let $c_0x+\b{c}^\intercal\b{y} \leq d$ define a facet of $\conv(S^n)$. If $\b{c} \ge \b{0}$ and $c_0=0$, then there exists a unique $i \in [n]$ such that $c_i > 0$. Consequently, the inequality is a positive scalar multiple of $y_i \leq 1$ for a unique $i\in[n]$.
\end{proposition}
\begin{proof}
    Let $I:= \{i \in [n] \mid c_i > 0\}$ be the index set of positive coefficients. Since $c_0 = 0$, the inequality reduces to
    \begin{equation} \label{eqn:y_only}
        \sum_{i \in I} c_i y_i \leq d.
    \end{equation}
    Because the inequality defines a facet, it is nontrivial, and hence $I\neq\emptyset$.
    
    Since the point $\left(0,\sum_{i\in I}\b{e}_i\right)$ belongs to $S^n$, validity of \eqref{eqn:y_only} implies $\sum_{i\in I}c_i\leq d$.
    Therefore, the facet-defining bound inequalities $y_i \leq 1$ imply \eqref{eqn:y_only} as
    \[
        \sum_{i\in I}c_i y_i \leq \sum_{i\in I}c_i \leq d.
    \]

    A facet-defining inequality of a full-dimensional polytope cannot be implied by a collection of facet-defining inequalities unless at least one of them defines the same facet. Consequently, \eqref{eqn:y_only} must be a positive scalar multiple of $y_i \leq 1$ for some $i \in I$. Since $c_j>0$ for every $j\in I$, this is possible only if $I=\{i\}$. Thus, there exists a unique
    $i\in[n]$ such that $c_i>0$, and \eqref{eqn:y_only} is a positive scalar multiple of $y_i\leq1$.
\end{proof}

The following result integrates the preceding propositions to provide the facet description of $\conv(S^n)$.

\begin{theorem} \label{thm:facets}
    Up to positive scalar multiplication, the facet-defining inequalities of $\conv(S^n)$ are precisely the bound inequalities \eqref{eqn:bdx} and \eqref{eqn:bdy}, together with the leap inequalities \eqref{eqn:leap} for all nonempty $Q = \{i_1, \ldots, i_q\} \subseteq [n]$ such that $a_{i_1} < a_{i_2} < \ldots < a_{i_q}$.
\end{theorem}

\begin{proof}
    Corollary~\ref{cor:lifting} and Proposition~\ref{prop:leap} show that every inequality in the statement defines a facet. It remains to prove that there are no others.

    Let $c_0x+\b{c}^\intercal\b{y}\leq d$ be an arbitrary facet-defining inequality. Proposition~\ref{prop:negprop} handles every case in which one of its coefficients is negative, and Proposition~\ref{prop:poscofacets-x} handles the case where $\b{c} \ge \b{0}$ and $c_0=0$. We may therefore assume that $\b{c} \geq \b{0}$ and $c_0>0$.

    Let $Q:=\{i\in[n]\mid c_i>0\}$. If $Q=\emptyset$, then the inequality is $c_0x\leq d$. Since the maximum value of $x$ over $\conv(S^n)$ is $1$ and the equality face is nonempty, $d=c_0$. Thus the inequality is a positive scalar multiple of $x\leq1$.

    We may therefore assume that $\emptyset \neq Q = \{i_1, \cdots, i_q\}$ with $i_1 < \cdots < i_q$. By Theorem~\ref{thm:lifting}, the reduced inequality
    \begin{equation}\label{eqn:reduced-facet}
        c_0x+(\b{c}^Q)^\intercal\b{u}\leq d
    \end{equation}
    defines a facet $F_Q$ of
    \[
        P_Q:=\conv(S^q(\b{a}^Q)).
    \]
    Let $\b{z}_Q^1,\ldots,\b{z}_Q^{q+1}$ be the points defined by \eqref{eqn:zk} for the parameter vector $\b{a}^Q$, and let $\bar{\b{z}}$ be an extreme point of $F_Q$. Then, $\bar{\b{z}}$ is also an extreme point of $P_Q$. By Part~\ref{lem:zk-1} of Lemma~\ref{lem:zk}, there exists $\ell \in [q+1]$ such that $\bar{\b{z}} \leq \b{z}_Q^\ell$. Since all coefficients in \eqref{eqn:reduced-facet} are strictly positive, if $\bar{\b{z}}\neq\b{z}_Q^\ell$, then
    \[
        \begin{bmatrix} c_0 & (\b{c}^Q)^\intercal \end{bmatrix} \b{z}_Q^\ell >
        \begin{bmatrix} c_0 & (\b{c}^Q)^\intercal \end{bmatrix} \bar{\b{z}} = d,
    \]
    contradicting the validity of \eqref{eqn:reduced-facet}. Therefore, every extreme point of $F_Q$ belongs to the set $\{\b{z}_Q^1,\ldots,\b{z}_Q^{q+1}\}$.

    Since $F_Q$ has dimension $q$, it contains at least $q+1$ affinely independent extreme points. It follows that $\b{z}_Q^1,\ldots,\b{z}_Q^{q+1}$ are affinely independent and all belong to $F_Q$. By Part~\ref{lem:zk-2} of Lemma~\ref{lem:zk}, the components of $\b{a}^Q$ are pairwise distinct. 
    By Part~\ref{lem:zk-3}, the same $q+1$ affinely independent points satisfy at equality the reduced-space leap inequality
    \[
        x+\sum_{k=1}^{q} \bigl(a_{i_{k+1}}-a_{i_{k-1}}\bigr)u_k \leq 1+a_{i_q},
    \]
    where $a_{i_0}=0$ and $a_{i_{q+1}}=1$. Hence its equality equation and that of \eqref{eqn:reduced-facet} define the same hyperplane in $\Re^{q+1}$, and one is a positive scalar multiple of the other. Lifting back to the original space through Theorem~\ref{thm:lifting}, we conclude that the original inequality is a positive scalar multiple of the leap inequality associated with $Q$.
\end{proof}

Theorem~\ref{thm:facets} gives a complete facet description of $\conv(S^n)$. Although the number of leap inequalities may grow exponentially with $n$, the next section shows that they admit a polynomial-time separation algorithm.

\section{Separation and numerical examples}\label{sec:sep}
In this section, we develop a polynomial-time separation procedure for the leap inequalities. By Theorem~\ref{thm:facets}, $\conv(S^n)$ is described by the bound inequalities and the leap inequalities. Hence, for a point $(x^\ast,\b{y}^\ast)\in[0,1]^{n+1}$, it suffices to separate over the family of leap inequalities. In exact arithmetic, the procedure either returns a violated leap inequality or certifies that $(x^\ast,\b{y}^\ast)\in\conv(S^n)$. Its numerical implementation uses a tolerance $\epsilon\geq0$ and certifies instead that no leap inequality is violated by more than $\epsilon$. Our construction is inspired by the separation framework of \cite{atamturk2000mixed} for the mixed vertex-packing problem.

Let $G=(\cN,\cA)$ be the directed graph with node set
\[
    \cN:=\{s\}\cup[n]\cup\{t\},
\]
where $s$ and $t$ are the source and sink, respectively, and arc set
\[
    \cA :=\{(s,j)\mid j\in[n]\} \cup \{(i,j)\mid 1\leq i<j\leq n,\ a_i<a_j\} \cup \{(i,t)\mid i\in[n]\}.
\]
By construction, $G$ is a directed acyclic graph (DAG). For each arc $(i,j)\in\cA$, define its length by
\begin{equation}\label{eq:arccost}
    \omega_{ij} :=
    \begin{cases}
        x^\ast, & i=s,\\
        a_jy_i^\ast - a_iy_j^\ast, & 1 \leq i < j \leq n,\\
        y_i^\ast-(1+a_i), & j=t.
    \end{cases}
\end{equation}

\begin{example}\label{ex:network}
    Let $n=4$ and $\b{a}=(0.2,0.3,0.3,0.5)$. The corresponding network is illustrated in Fig.~\ref{fig:leapnetwork}. In particular, $(2,3)$ is not an arc because $a_2=a_3$.

    \begin{figure}[htbp]
        \centering
        \begin{tikzpicture}
        \tikzset{
            EdgeLabelStyle/.style={
                fill=none,
                draw=none,
                text opacity=1
            }
        }
        \Vertex[label=s]{0}
        \Vertex[x=2,y=2,label=1]{1}
        \Vertex[x=2,y=-2,label=2]{2}
        \Vertex[x=5,y=2,label=3]{3}
        \Vertex[x=5,y=-2,label=4]{4}
        \Vertex[x=7,label=t]{5}
        \Edge[label=$x^\ast$,position={above left=0.5mm},Direct](0)(1)
        \Edge[position={below left=0.5mm},Direct](0)(2)
        \Edge[position={above left},Direct](0)(3)
        \Edge[position={left=1mm},Direct](0)(4)
        \Edge[position={left},Direct](1)(2)
        \Edge[label=$a_3y_1^\ast-a_1y_3^\ast$,position={above},Direct](1)(3)
        \Edge[position={above right=0.5mm},Direct](1)(4)
        \Edge[position={above right},Direct](1)(5)
        \Edge[position=below,Direct](2)(4)
        \Edge[position={above right=1mm},Direct](2)(5)
        \Edge[position=right,Direct](3)(4)
        \Edge[label=$y_3^\ast-(1+a_3)$,position={above right},Direct](3)(5)
        \Edge[position={below=0.5mm},Direct](4)(5)
        \end{tikzpicture}
        \caption{Network in Example~\ref{ex:network}; selected arc lengths are displayed.}
        \label{fig:leapnetwork}
    \end{figure}
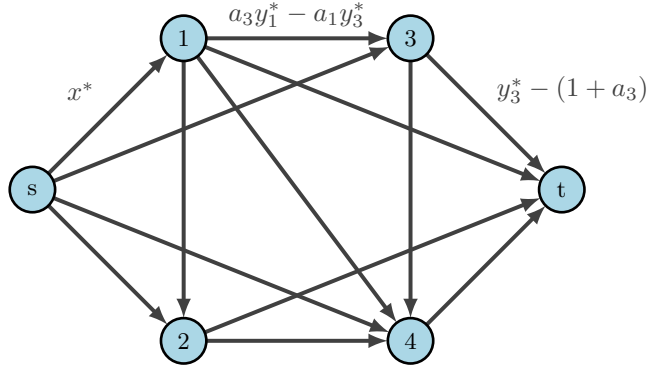
\end{example}

\begin{proposition}\label{prop:leapseparation}
    For every $(x^\ast,\b{y}^\ast)\in[0,1]^{n+1}$, the maximum violation of a leap inequality at $(x^\ast,\b{y}^\ast)$ equals the length of a longest $s$--$t$ path in $G$. Consequently, a most violated leap inequality can be found in $\mathcal{O}(n^2)$ time. In exact arithmetic, $(x^\ast, \b{y}^\ast) \in \conv(S^n)$ if and only if the longest $s$--$t$ path has length at most zero.
\end{proposition}

\begin{proof}
    Consider a nonempty set $Q = \{i_1, \ldots, i_q\} \subseteq [n]$ with $a_{i_1} < \ldots < a_{i_q}$. The corresponding path
    \[
        s \to i_1 \to i_2 \to \cdots \to i_q \to t
    \]
    has length
    \begin{align*}
        \omega_{s,i_1} + \sum_{k=1}^{q-1}\omega_{i_k,i_{k+1}} + \omega_{i_q,t}
        &= x^\ast + \sum_{k=1}^{q-1} \left(a_{i_{k+1}}y_{i_k}^\ast-a_{i_k}y_{i_{k+1}}^\ast\right) +y_{i_q}^\ast-(1+a_{i_q}) \\
        &= x^\ast + \sum_{k=1}^{q} \left(a_{i_{k+1}}-a_{i_{k-1}}\right) y_{i_k}^\ast - (1+a_{i_q}),
    \end{align*}
    where $a_{i_0}:=0$ and $a_{i_{q+1}}:=1$, as in Section~\ref{sec:leap}. This is exactly the violation, defined as the left-hand side minus the right-hand side, of the leap inequality associated with $Q$.

    Conversely, the internal nodes of every $s$--$t$ path form a sequence $i_1<\cdots<i_q$ satisfying $a_{i_1}<\cdots<a_{i_q}$ and hence determine a unique admissible set $Q$. Thus, admissible sets $Q$ are in one-to-one correspondence with $s$--$t$ paths, and the length of each path equals the violation of the associated leap inequality.

    Since $G$ is acyclic, a longest path can be found in $\mathcal O(|\cN|+|\cA|)$ time by processing the vertices in a topological order; see, for example, Proposition~T in \cite[Section~4.4]{sedgewick2011algorithms}.
    The graph has $n+2$ vertices and $\mathcal O(n^2)$ arcs. Hence, a most violated leap inequality can be found in $\mathcal O(n^2)$ time.
    The final equivalence follows from Theorem~\ref{thm:facets}, since $(x^\ast,\b{y}^\ast) \in [0,1]^{n+1}$ already satisfies all bound inequalities.
\end{proof}

Algorithm~\ref{alg:leapseparate} implements the separation procedure. A path with a single internal node corresponds to $Q=\{i\}$ and therefore checks the projected McCormick inequality $x+y_i \leq 1+a_i$.

\begin{algorithm}[htbp]
\caption{Separation of leap inequalities (\textsc{leapSep})}\label{alg:leapseparate}
\begin{algorithmic}[1]
\Require $n$, $\b{a}$, $(x^\ast,\b{y}^\ast)\in[0,1]^{n+1}$, and tolerance $\epsilon\geq0$
\State $\cN\gets\{s\}\cup[n]\cup\{t\}$
\State $\cA\gets\{(s,j)\mid j\in[n]\}\cup\{(i,j)\mid 1\leq i<j\leq n,\ a_i<a_j\}\cup\{(i,t)\mid i\in[n]\}$
\For{each $(i,j)\in\cA$}
    \If{$i=s$}
        \State $\omega_{ij}\gets x^\ast$
    \ElsIf{$j=t$}
        \State $\omega_{ij}\gets y_i^\ast-(1+a_i)$
    \Else
        \State $\omega_{ij}\gets a_jy_i^\ast-a_iy_j^\ast$
    \EndIf
\EndFor
\State Compute a longest $s$--$t$ path $\mathcal{L}$ and its length $\eta^\ast$
\If{$\eta^\ast\leq\epsilon$}
    \State \textbf{return} no leap inequality is violated by more than $\epsilon$
\Else
    \State $Q\gets\{i_1,\ldots,i_q\}$, where $i_1,\ldots,i_q$ are the internal nodes of $\mathcal{L}$
    \State \textbf{return} $Q$, $\eta^\ast$, and the leap inequality associated with $Q$
\EndIf
\end{algorithmic}
\end{algorithm}

We conclude with two examples illustrating the strengthening obtained from leap inequalities. In both examples, the constraints $xy_i\leq a_i$ induce a copy of $S^n$ in the $(x,\b{y})$-space. To obtain the leap-enhanced bounds, we repeatedly solve the current relaxation and apply Algorithm~\ref{alg:leapseparate} to an optimal solution. Whenever the maximum violation exceeds the prescribed tolerance $\epsilon$, we add the corresponding most violated leap inequality and solve the resulting relaxation again. The process terminates when no leap inequality is violated by more than $\epsilon$. In these two examples, the separation tolerance is set to $\epsilon = 10^{-6}$, and the relaxations are solved to optimality using Gurobi 12.0.3 with default settings. The number of cuts reported below is the total number of leap inequalities added.

For a maximization problem, let $z_{\mathrm{Mc}}$ denote the McCormick relaxation bound, $z_{\mathrm{rel}}$ the bound obtained from a strengthened relaxation, and $z_{\mathrm{glob}}$ the globally optimal value. We report the percentage of the McCormick gap closed by the strengthened relaxation:
\begin{equation}\label{eq:gapclosed}
    \operatorname{GapClosed}(\%) := \frac{z_{\mathrm{Mc}}-z_{\mathrm{rel}}}{z_{\mathrm{Mc}}-z_{\mathrm{glob}}} \cdot 100\%.
\end{equation}

\begin{example} \label{eg1}
    Consider
    \begin{equation*}
    \begin{aligned}
        \max \quad  & 25xy_2+10x+10y_1-10y_2+17.5y_3 \\
        \text{s.t.} \quad & xy_1 \leq \frac18, \qquad xy_2 \leq \frac14, \qquad xy_3 \leq \frac12, \\
                    & (x,\b{y}) \in [0,1]^4.
    \end{aligned}
    \end{equation*}
    
    Introduce auxiliary variables $w_i=xy_i$, $i\in[3]$, and let $\b{a} = (1/8, 1/4, 1/2)$. After replacing every occurrence of $xy_i$ by $w_i$, including in the objective, the baseline LP relaxation imposes, for each $i\in[3]$,
    \begin{equation}\label{eq:example-mccormick}
        0 \leq w_i, \qquad  w_i \leq x, \qquad  w_i \leq y_i, \qquad
        w_i \geq x+y_i-1, \qquad    w_i \leq a_i.
    \end{equation}
    These are the McCormick inequalities together with the upper bound inherited from $xy_i\leq a_i$.

    As shown in \cite{anstreicher2021convex}, for $a\in(0,1)$, the convex hull of
    \[
        \{(x,y,w) \mid w=xy, (x,y) \in [0,1]^2, w\leq a\}
    \]
    is described by the McCormick inequalities, $w\leq a$, and the rotated second-order-cone inequality
    \[
        w^2 \leq axy.
    \]
    We therefore strengthen the baseline by imposing $w_i^2 \leq a_ixy_i$ for each $i \in [3]$. We compare the baseline with the SOC-strengthened relaxation, the relaxation augmented with separated leap inequalities, and the relaxation containing both strengthenings.

    The globally optimal value is $28.75$. Table~\ref{tab:eg1} reports the relaxation bounds, the number of leap cuts added, and the gap closed according to \eqref{eq:gapclosed}. The SOC and leap strengthenings close $47.92\%$ and $33.33\%$ of the McCormick gap, respectively, while their combination closes $86.98\%$.

    \begin{table}[htbp]
        \centering
        \small
        \begin{tabular}{lcccc}
            \toprule
            & McCormick & $+$ SOC & $+$ leap & $+$ SOC $+$ leap\\
            \midrule
            Relaxation bound & 32.50 & 30.70 & 31.25 & 29.24\\
            Leap cuts added  & --    & --    & 1     & 1\\
            Gap closed (\%)  & --    & 47.92 & 33.33 & 86.98\\
            \bottomrule
        \end{tabular}
        \caption{Comparison of the relaxations for Example~\ref{eg1}.}
        \label{tab:eg1}
    \end{table}
\end{example}

\begin{example} \label{eg2}
    Consider
    \begin{equation*}
    \begin{aligned}
        \max \quad  & 30xy_2+16x+14y_1-2y_2+21y_3-4y_4 \\
        \text{s.t.} \quad   & x+10y_4 \leq 7.5, \\
                    & \sum_{i=1}^4xy_i+x+\sum_{i=1}^4y_i\leq3.5, \\
                    & xy_1 \leq \frac18, \qquad xy_2 \leq \frac14, \qquad
                      xy_3 \leq \frac12, \qquad xy_4 \leq \frac34, \\
                    & (x,\b{y}) \in [0,1]^5.
    \end{aligned}
    \end{equation*}

    We introduce $w_i=xy_i$, $i\in[4]$, and replace every occurrence of $xy_i$ by $w_i$, including those in the objective and the second coupling constraint. We then impose the baseline formulation \eqref{eq:example-mccormick} with $\b{a}=(1/8,1/4,1/2,3/4)$, together with the two linear constraints. We compare the same four relaxations as in Example~\ref{eg1}.

    The globally optimal value is $39$. The results are reported in Table~\ref{tab:eg2}. The SOC and leap strengthenings close $8.70\%$ and $34.78\%$ of the McCormick gap, respectively, while their combination closes $63.15\%$. 

    \begin{table}[htbp]
        \centering
        \small
        \begin{tabular}{lcccc}
            \toprule
            & McCormick & $+$ SOC & $+$ leap & $+$ SOC $+$ leap \\
            \midrule
            Relaxation bound & 44.75 & 44.25 & 42.75 & 41.12 \\
            Leap cuts added  & --    & --    & 1     & 2 \\
            Gap closed (\%)  & --    & 8.70  & 34.78 & 63.15 \\
            \bottomrule
        \end{tabular}
        \caption{Comparison of the relaxations for Example~\ref{eg2}.}
        \label{tab:eg2}
    \end{table}
\end{example}

The examples indicate that neither strengthening dominates the other. The SOC constraints strengthen the relaxation by characterizing the convex hull of each individual bilinear term under an upper bound, whereas the leap inequalities tighten the relaxation by capturing structural interactions among multiple bilinear terms.

\section{Conclusion}
In this paper, we studied the convexification of bilinear sets in which the terms share a common variable and are subject to individual upper bounds. We derived complete vertex and facet descriptions of $\conv(S^n)$, thereby providing a full polyhedral characterization of the convex hull. Both descriptions exploit the ordering and values of the bilinear upper bounds.

The facet description consists of the bound inequalities together with a structured family of leap inequalities that capture interactions among multiple bilinear terms. By exploiting the dependence induced by the common variable, these inequalities can strengthen the standard McCormick relaxation. Although the family may contain exponentially many inequalities, we showed that a most violated leap inequality can be found in polynomial time through a longest-path problem on a directed acyclic graph. Our numerical examples further illustrate that leap inequalities and second-order cone constraints can provide complementary improvements to convex relaxation bounds.

These results provide new tools for constructing stronger relaxations of bipartite bilinear programs. Future research directions include extending the analysis to sets involving multiple shared variables or more general bilinear interaction structures. It would also be valuable to integrate the separation procedure into spatial branch-and-bound algorithms for general bipartite bilinear programs and to assess, through a broader computational study, the extent to which dynamically generated leap inequalities improve relaxation quality and overall solution performance.

\section*{Acknowledgments} 
The authors thank Dr. Margaret Wiecek for helpful discussions and insights during the early stages of this work. The authors also thank Yachen Liu and Dr. Weikun Chen for insightful discussions and for generously sharing ideas and references that helped guide the development of the separation algorithm.

\section*{Funding}
The second author was partially supported by the Air Force Office of Scientific Research (AFOSR) under Grant No.~FA9550-23-1-0508.

\section*{Data Availability}
No external datasets were used. All data needed to reproduce the numerical examples are contained in the article.

\bibliography{ref.bib}
\end{document}